\documentclass[a4paper]{article}
\usepackage{amsmath, amsthm, amsfonts, amssymb, mathrsfs}
\usepackage[english]{babel}
\usepackage[all]{xy}
\usepackage[pagebackref]{hyperref}
\usepackage{cite}

\theoremstyle{definition}
\newtheorem{theorem}{Theorem}[section]
\newtheorem{proposition}[theorem]{Proposition}
\newtheorem{corollary}[theorem]{Corollary}
\newtheorem{lemma}[theorem]{Lemma}
\newtheorem{remark}[theorem]{Remark}
\newtheorem{definition}[theorem]{Definition}
\newtheorem{example}[theorem]{Example}
\newtheorem{question}[theorem]{Question}
\newtheorem*{acknowledgements}{Acknowledgements}

\newcommand{\CC}{\mathbb{C}}
\newcommand{\PP}{\mathbb{P}}
\newcommand{\UUb}{\mathbb{U}}
\newcommand{\NN}{\mathbb{N}}
\newcommand{\OO}{\mathcal{O}}
\newcommand{\UU}{\mathcal{U}}
\newcommand{\FF}{\mathcal{F}}
\newcommand{\LL}{\mathcal{L}}

\newcommand{\ZZ}{\mathcal{Z}}
\newcommand{\BB}{\mathcal{B}}
\newcommand{\KK}{\mathcal{K}}

\newcommand{\II}{\mathscr{I}}
\newcommand{\IIc}{\mathscr{C}}
\newcommand{\pp}{\mathfrak{p}}

\newcommand{\sing}{\mathrm{sing}}
\newcommand{\codim}{\mathrm{codim}}
\newcommand{\ann}{\mathrm{ann}}
\newcommand{\proj}{\text{Proj}}

\title{The Kupka Scheme and Unfoldings}
\author{C\'esar Massri\footnotemark[1], Ariel Molinuevo\footnotemark[1] ,  Federico Quallbrunn.\thanks{The author was fully supported by CONICET, Argentina.}}
\date{}

\begin{document}

\maketitle

\centerline{\emph{Department of Mathematics, FCEyN, University of Buenos Aires, Argentina.}}

\bigskip

\begin{abstract}
Let $\omega$ be a differential 1-form defining an algebraic foliation of codimension 1 in projective space. In this article we use commutative algebra to study the singular locus of $\omega$ through its ideal of definition. Then, we expose the relation between the ideal defining the Kupka components of the singular set of $\omega$ and the first order unfoldings of $\omega$.
Exploiting this relation, we show that the set of Kupka points of $\omega$ is generically not empty.

As an application of these results, we can compute the ideal of first order unfoldings for some known components of the space of foliations.

\end{abstract}

\section{Introduction}

An \emph{algebraic foliation of codimension one} in projective space $\PP^n$ over $\CC$, is given by a global section $\omega$ of
the sheaf of twisted differential 1-forms $\Omega^1_{\PP^n}(e)$ that verify the \emph{Frobenius integrability condition} $\omega\wedge d\omega = 0$.
The space of such foliations forms a projective variety $\FF^1(\PP^n,e)$. A first invariant that one can attach to a codimension one foliation is its degree,
which is given by the number of tangencies of a generic line with the foliation.
For $\omega\in\FF^1(\PP^n,e)$ the degree is known to be $e-2$.

\bigskip

The \emph{singular locus} of a foliation given by $\omega$ is defined by $\sing(\omega)_{set}=\{p\in\PP^n :\ \omega(p)=0 \}$. It can be decomposed as a union
\[
\sing(\omega)_{set}= \KK_{set}\cup\LL_{set}
\]
where $\KK_{set}$ is the closure of the set of \emph{Kupka points} which are the singular points of $\omega$ such that $d\omega\neq 0$; and $\LL_{set}$ is defined as the closure of $\sing(\omega)_{set}\setminus\KK_{set}$. We append a subindex \emph{set} to stress the fact that this is a set-theoretical approach. Kupka points were first studied by Ivan Kupka in \cite{kupka}, where he first noted that the existence of such points is stable under deformations of $\omega$, see also \cite[Chapter 1.4, p.~38]{aln}. Sometimes the subvariety $\KK_{set}$ is referred to as the \emph{Kupka component}. Locally around each Kupka point it is a smooth variety of codimension 2. Also $\omega$ has locally a normal form around Kupka points. On the other side, if $\LL_{set}$ has codimension greater than 3, by B. Malgrange's theorem in \cite[Th\'eor\`eme 0.1, p.~163]{m-f1}, $\omega$ admits locally around each point of $\LL_{set}$, an analytic integrating factor.

\medskip
Since a series of articles which appeared around 1994 such as \cite{omegar-marcio,ballico,cerveau-lins-neto-kupka}, a lot of attention had been paid to foliations in $\PP^n$ with a non-empty Kupka variety. Many results on such foliations focus around special cases where $\KK_{set}$ is non-singular and \emph{every} point in $\KK_{set}$ is a Kupka point. In this setting there are results such as the main theorem of \cite{cerveau-lins-neto-kupka}, stating that if $\KK_{set}$ is \emph{globally} a smooth complete intersection then $\omega$ has a meromorphic first integral. The work of D. Cerveau and A. Lins Neto motivated the question of when $\KK_{set}$ is a non-singular global complete intersection subvariety of $\PP^n$. In this regard, the first results are the very important principal theorems of \cite{omegar-marcio} and its addendum \cite{ballico}.
\par Other results on properties of a foliation with a non-singular Kupka variety are the main subject of papers like \cite{omegar-kupka,scardua-camacho}, and more recently \cite{omegar-ivan}.

\bigskip

In this work we address a more fundamental question on Kupka singularities, namely: Which forms $\omega$ on $\PP^n$ admit a \emph{non-empty} Kupka variety? In every known irreducible component of the space of integrable forms, a generic element has indeed Kupka points. Whether this is a general situation or a coincidence remains unknown.

We find a partial answer to this question which takes us to consider the \emph{schematic} structure of $\sing(\omega)$ given by the homogeneous ideal generated by the coefficients of $\omega$. See Theorem \ref{teo3} for a full statement of the following result.

\begin{theorem}
If the ideal of $\sing(\omega)$ is radical, then its Kupka variety is non-empty.
\end{theorem}

\smallskip In order to prove this statement, we need a result of independent interest, namely Theorem \ref{division}, which is a local division property for $d\omega$.

\bigskip
Moreover, we also relate the algebraic structure of the ideal defining the Kupka variety of $\omega$ with its \emph{first order unfoldings}.

\medskip

A first order unfolding of $\omega$ is given by an integrable differential 1-form $\widetilde{\omega}_\varepsilon$, defined in the scheme $\PP^n[\varepsilon]:=\PP^n\times spec(k[\varepsilon])$, where $k[\varepsilon]=k[x]/(x^2)$, such that $\widetilde{\omega}_\varepsilon$ reduces to $\omega$ when intersected with the central fiber $\PP^n$.
The set $U(\omega)$ of first order unfoldings of $\omega$ has a natural vector space structure.
After \cite[Definition 4.10, p.~193]{suwa-unfoldings}, or \cite[Def. 2.2.5, p.~7]{moli} for a more algebraic approach, we say that two unfoldings $\widetilde{\omega}_\varepsilon$ and $\widetilde{\omega}_\varepsilon'$ are \emph{isomorphic} if there is an isomorphism $\phi$ of $\PP^n[\varepsilon]$ such that $\phi$ restricts to the identity in the central fiber and $\phi^*\widetilde{\omega}_\varepsilon= \widetilde{\omega}_\varepsilon'$.

First order unfoldings of a form are closely related to its \emph{first order deformations}.
A first order deformation of $\omega$ is given by a family of differential 1-forms $\omega_\varepsilon$, parameterized by an infinitesimal parameter $\varepsilon$, such that $\omega_\varepsilon$ is integrable and reduces to $\omega$ when $\varepsilon=0$. These are the `classic' perturbations and they identify with the Zariski tangent space $T_\omega\FF^1(\PP^n,e)$.
They relate to unfoldings through the exact sequence
\[
\xymatrix@R=0pt{
0\ar[r] & IF(\omega) \ar[r] & U(\omega) \ar[r] & D(\omega),
}
\]	
where $IF(\omega)$ denotes the \emph{integrating factors} of $\omega$,
\[
IF(\omega)=\{ f\in H^0\left(\OO_{\PP^n}(e)\right)\,\colon\, fd\omega=-\omega\wedge df\}.
\]

\medskip

The theory of unfoldings for differential forms was developed by Tatsuo Suwa in \cite{suwa-unfoldings}. Let us denote by $\OO_{\CC^{n+1},p}$ and $\Omega^1_{\CC^{n+1},p}$ the analytic germs of functions and differential 1-forms around $p\in\CC^{n+1}$, respectively. If $\varpi\in \Omega^1_{\CC^{n+1},p}$ defines a foliation, the space of unfoldings of $\varpi$ can be parameterized as
\begin{equation*}
U_p(\varpi) = \left\{(h,\eta)\in \OO_{\CC^{n+1},p}\times\Omega^1_{\CC^{n+1},p}:\ h\ d\varpi= \varpi\wedge (\eta - dh)\right\}\big/\CC.(0,\varpi).
\end{equation*}
For a generic $\varpi$, the projection of $U_p(\varpi)$ to the first coordinate defines an ideal $I_p(\varpi)\subseteq \OO_{\CC^{n+1},p}$. This ideal gives a good algebraic structure to study $U_p(\varpi)$ and was used by Suwa to classify first order unfoldings of rational and logarithmic foliations, see \cite{suwa-meromorphic,suwa-multiform}. We refer the reader to \cite{suwa-review} for a review of his work.

\medskip

For $\omega\in \FF^1(\PP^n,e)$, first order unfoldings can be parameterized in an analogous way as
\begin{equation*}
\left\{(h,\eta)\in H^0(\PP^n,\OO_{\PP^n}(e))\times H^0(\PP^n,\Omega^1_{\PP^n}(e))\,\colon\, h\ d\omega  = \omega\wedge (\eta - dh)\right\}\big/\CC.(0,\omega).
\end{equation*}
Since $U(\omega)$ is a finite dimensional vector space there is no ideal associated to it. To remedy this shortcoming one can proceed as follows. Let $S=\CC[x_0,\ldots,x_n]$ be the ring of homogeneous coordinates in $\PP^n$ and consider $\omega$ as an affine differential form in $\CC^{n+1}$, the cone of $\PP^n$. Then we recall from \cite{moli} the $S$-module of \emph{graded projective unfoldings},
\[
\UUb(\omega) = \left\{(h,\eta)\in S\times \Omega^1_{S}: \ L_R(h)\ d\omega = L_R(\omega)\wedge(\eta - dh) \right\}\big/ S.(0,\omega).
\]
where $L_R$ is the Lie derivative with respect to the radial vector field $R=\sum_{i=0}^nx_i\frac{\partial}{\partial x_i}$, see Definition \ref{unf-unf}.

\medskip

The projection of $\UUb(\omega)$ to the first coordinate defines an ideal $I(\omega)\subseteq S$ emulating the situation in the local analytic setting. We will call $I(\omega)$ the \emph{ideal of graded projective unfoldings} of $\omega$, or simply, the \emph{ideal of unfoldings of $\omega$} if no confusion can arise. As we will show later in Proposition \ref{I/J}, the \emph{classes of isomorphism of graded unfoldings} of $\omega$ can be computed by a quotient of $I(\omega)$.

\medskip

To achieve a deeper understanding of $\sing(\omega)$, the varieties $\KK_{set}$ and $\LL_{set}$ had to be redefined as subschemes $\KK$ and $\LL$, respectively. See Section \ref{teo} for these two definitions, as well as for an example showing that the reduced structure of $\KK$ might differ from $\KK_{set}$. Two of our main results show how to relate the module of isomorphism classes of unfoldings with the homogeneous coordinate ring of $\LL$, generalizing \cite[Theorem 5.1.4, p.~18]{moli}. Another of our main results states the relation between the ideals $I(\omega)$  and the graded ideal of $\KK$.
We summarize these results bellow and refer the reader to Theorem \ref{teo1} and Corollary \ref{prop1}
for complete statements.

\begin{theorem}
Let $\omega\in\FF^1(\PP^n,e)$ be a generic foliation
and denote by $K$ and $L$ the ideals associated to $\KK$ and $\LL$, respectively. Then
\[
\sqrt{I(\omega)}=\sqrt{K}.
\]
Even more so, if $K$ and the ideal of $\sing(d\omega)$ are comaximal, then there is an isomorphism of $S$-modules
\[
I(\omega)\big/J(\omega)\cong S\big/ L,
\]
where $J(\omega)$ is the singular ideal of $\omega$.
\end{theorem}

\medskip

In Section \ref{teo} we also give a series of specific examples exposing different situations. The computations involved in such examples were done by using the computer algebra software Macaulay 2  together with the differential algebra package DiffAlg, see \cite{M2} and \cite{DiffAlgM2}, respectively.

\bigskip

Finally, in Section \ref{app} we apply the previous results to pullback and split tangent sheaf foliations and compute their unfoldings ideal. Let us recall their definitions.

Given a dominating homogeneous map $\xymatrix@1{F:\PP^n\ar@{-->}[r] & \PP^2}$ and a differential form $\omega$ in $\PP^2$, the pullback $F^*\omega$ defines an integrable differential form in $\PP^n$. The set of such pullbacks, for fixed degrees on $F$ and $\omega$, defines an irreducible component of the space of foliations, as it is shown in \cite{pullback}.

A foliation with split tangent sheaf in $\PP^n$ can be written as
\[
\omega =i_R i_{X_1}\cdots i_{X_{n-1}} dx_0\wedge\ldots\wedge dx_n,
\]
where $X_1,\ldots, X_{n-1}$ are vector fields and $R$ is the radial vector field. These foliations are a generalization of foliations associated to affine Lie algebras, which where studied in \cite{omegar2}.
They form an irreducible component of the space of foliations, see \cite{fj}.

\begin{acknowledgements}
We would like to thank Fernando Cukierman and Tatsuo Suwa for their valuable comments and suggestions.
\end{acknowledgements}

\section{Codimension one foliations}\label{foliations}

Along this section we first give basic definitions for foliations in $\PP^n$. Then, we state and prove a division lemma for integrable forms on smooth varieties.

\medskip

Let us denote $\Omega^1_{\PP^n}(e)$ the sheaf of twisted differential 1-forms in $\PP^n$ of degree $e$.

\begin{definition} We will say that a generically rank 1 subsheaf $\FF$ of $\Omega^1_{\PP^n}(e)$, $e\geq 2$ is an
\emph{algebraic foliation of codimension 1} on $\PP^n$, a \emph{foliation} from now on, if $\FF$ is generated by a non zero global section
$\omega\in H^0(\PP^n,\Omega^1_{\PP^n}(e))$ such that
$\omega\wedge d\omega = 0$. We recall from the introduction that such foliations have degree $e-2$.
\end{definition}

\medskip

A foliation is required to have singular locus of codimension greater than 2. As we will show below, this is equivalent to ask that $\omega$ is not of the form $f.\omega'$, for some global section $f\in H^0(\PP^n,\OO_{\PP^n}(d))$
and a 1-form $\omega'\in H^0(\PP^n,\Omega^1_{\PP^n}(e-d))$. Also, integrable differential 1-forms define the same foliation up to scalar multiplication. Then, we will denote the set of codimension 1 foliations of degree $e-2$ as
\[
\FF^1(\PP^n,e) := \left\{\omega\in\PP\left(H^0(\PP^n,\Omega^1_{\PP^n}(e))\right):\ \omega\wedge d\omega=0,\ \codim(\sing(\omega))\geq 2 \right\}.
\]

We can give to $\FF^1(\PP^n,e)$ a subscheme structure defined by the equations $\omega\wedge d\omega = 0$.

As we are going to fix one generator for each foliation we might refer simply as $\omega$ to the foliation $\FF=(\omega)$.

\bigskip

Let us denote by $S$ the ring of homogeneous coordinates $\CC[x_0,\ldots,x_n]$. Then, a foliation defined by $\omega$ can be written as
\begin{equation}\label{foliations-omega}
 \begin{aligned}
  \omega=\sum_{i=0}^n A_i dx_i
 \end{aligned}
\end{equation}
where the $A_i$'s are homogeneous polynomials of degree $e-1$ that verify the integrability condition $\omega\wedge d\omega=0$ and the property of descent to projective space. The latter condition can be stated as the vanishing of the contraction of $\omega$ with the radial field $R=\sum x_i\frac{\partial}{\partial x_i}$.

\

By eq. \ref{foliations-omega}, the ideal of the singular locus of $\omega$ is given by
\[
\IIc(\omega) := (A_0,\ldots, A_n).
\]
From now on we will denote by $\IIc(\eta)$ the ideal generated by the polynomial coefficients of the differential form $\eta\in\Omega^r_S$. Note that this ideal may not be radical.

\bigskip

The \emph{Koszul complex} associated  to $\omega$, noted with $Kosz^\bullet(\omega)$, is defined as
\[
\xymatrix@C=30pt{
Kosz^\bullet(\omega): & S\ar[r]^-{\omega\wedge} & \Omega^1_{S} \ar[r]^-{\omega\wedge} & \Omega^2_{S} \ar[r]^-{\omega\wedge} & \ldots &
}
\]
We will usually denote the $p$-th homology group of this complex as $H^p(\omega)$.
The homology of $Kosz^\bullet(\omega)$ is able to compute the codimension of the singular set of $\omega$, by the well known result:
\begin{theorem}\label{koszul-teo}
For $\omega\in H^0(\PP^n,\Omega^1_{\PP^n}(e))$ the following are equivalent:
\begin{enumerate}
\item $\codim(\sing(\omega))\geq k$
\item $H^l(\omega)= 0$ for all $l<k$
\end{enumerate}
\end{theorem}
\begin{proof}
 See \cite[Appendix, p.~172]{m-f1} or \cite[Appendix, p.~87]{m-f2} for two proofs with different level of generalities in the local analytic  setting and \cite[Theorem 17.4, p.~424]{eisenbud} for a purely algebraic proof of our statement. 

\end{proof}

\bigskip

Note that $H^1(\omega)=0$ is equivalent to have $\codim(\sing(\omega))\geq 2$, as we asked in the definition of foliation. The integrability condition makes $d\omega$ to be a 2-cycle of $Kosz^\bullet(\omega)$ whose
homology class is non-trivial. This can be easily seen by comparing the degrees of the polynomial
coefficients of $d\omega$ and $\omega\wedge\eta$, for some differential 1-form $\eta$.

\medskip

Then, an algebraic foliation can be defined by a form $\omega\in\FF^1(\PP^n,e)$ 
with codimension 2 singular locus. In Theorem \ref{teo3} we will show that every foliation with reduced singular locus has Kupka points.

\subsection{Division Lemma}

Here we will prove a statement (Theorem \ref{division}) that will be applied later in Section \ref{teo} to foliations on projective space.
However, it is a result of independent interest which works in a wider context.

In the following we will do our computations in a non-singular variety $X$.
We will consider a $1$-form $\omega$ on $X$ with singular locus of codimension equal to or greater than 2.
And we will denote by $\mathcal{J}$ the ideal sheaf of $\sing(\omega)$.

Set $\ZZ^p(\omega)$ to be the module of degree $p$ cycles in the Koszul complex of $\omega$.
Set $D(\omega):=\ann(\omega)\subseteq TX$ the subsheaf of vector fields $\xi$ such that $\omega(\xi)=0$.
For a sheaf $F$ we denote by $F^\vee$ the (dual) sheaf $\mathcal{H}om(F,\OO_X)$.
We define with $N(\omega)$ the cokernel in the short exact sequence
\[
0\to D(\omega)\to TX\to N(\omega)\to 0
\]

As explained in \cite[section 4]{quallbrunn}, we have an exact sequence

\[
0\to \Omega^1_X\big/\OO_X\cdot(\omega)\to D(\omega)^\vee \to \mathcal{E}xt^1_X(N(\omega),\OO_X)\to 0
\]

We define a morphism $\Phi: \ZZ^2(\omega)\to D(\omega)^\vee$ in the following manner:
Let $\theta\in \ZZ^2(\omega)$, and $\xi\in D(\omega)$ , since $\theta\wedge \omega=0$, and $i_\xi \omega=0$ we have
\[
i_\xi\theta\wedge \omega= i_\xi (\theta\wedge\omega)=0.
\]
Then, as $H^1(\omega)=0$ we must have a unique $f\in \OO_X$ such that $i_\xi \theta= f\omega$.
We define $\Phi(\theta):D(\omega)^\vee\to \OO_X$ to be the $\OO_X$-linear map such that, to each $\xi\in D(\omega)^\vee$ assigns $\Phi(\theta)(\xi)=f$, where $i_\xi\theta = f\omega$.

\begin{lemma}
The morphism $\Phi:\ZZ^2(\omega)\to D(\omega)^\vee$ is injective.
\end{lemma}
\begin{proof}
Take a point $p\in X$ such that $\omega\otimes k(p)\neq 0$. Then, locally in $p$, there is a $1$-form $\eta$ such that $\theta=\omega\wedge\eta$.
Therefore, for each $\xi\in D(\omega)_p$, we have $i_\xi \theta = (i_\xi\eta) \omega$.
Now take a $\theta\in \ZZ^2(\omega)$ such that $\Phi(\theta)=0$; that means that, near $p$, $i_\xi\eta=0$ for every $\xi\in D(\omega)_p$. Then, there is a $g$ such that $\eta=g\cdot \omega$, and so $\theta_p=0$.
As this would happen for every $p$ in the dense open subset where $\omega\otimes k(p)\neq 0$, we have that if $\Phi(\theta)=0$ then $\theta=0$.

\end{proof}

Then we have a diagram

\[\xymatrix{
0 \ar[r] & \Omega^1_X/\OO_X\cdot(\omega) \ar[r] \ar@{=}[d] & \ZZ^2(\omega) \ar[r] \ar@{^{(}->}[d]^{\Phi} & H^2(\omega) \ar[r] \ar@{^{(}->}[d]& 0 \\
0 \ar[r]& \Omega^1_X/\OO_X\cdot(\omega) \ar[r]& D(\omega)^\vee \ar[r]& \mathcal{E}xt^1_X(N(\omega),\OO_X) \ar[r]& 0
}
\]

We thus arrive at another proof of \cite[Proposition 4, IV-7]{serre}.
First we are going to make clear the convention we are going to use when we mention the \emph{support} of a coherent sheaf.
Given a coherent sheaf $\FF$ over a scheme $X$ we denote with  $\mathrm{supp}(\FF)$ the subscheme of $X$ defined by the \emph{annihilator} of $\FF$, that is, the sheaf of ideals of $X$ such that in every $x\in X$ is locally given by the annihilator ideal $\mathrm{Ann}(\FF_x)$ of $\FF_x$ in $\OO_{X,x}$.
\begin{proposition}\label{sup}
With the notation as above we have $\mathrm{supp}(H^2(\omega))\subseteq \sing(\omega)$.
\end{proposition}
\begin{proof}
By the diagram above we have $\mathrm{supp}(H^2(\omega))\subseteq \mathrm{supp}(\mathcal{E}xt^1_X(N(\omega),\OO_X)$.
As explained in \cite[Remark 7.7, p.~178]{quallbrunn} the annihilator of $\mathcal{E}xt^1_X(N(\omega),\OO_X)$ is locally defined by the coefficients of $\omega$. 
So the subscheme $\mathrm{supp}(\mathcal{E}xt^1_X(N(\omega),\OO_X)$ is just $\sing(\omega)$.
\end{proof}

\begin{proposition}\label{propDiv}
Suppose $\theta \in \mathcal{J}\cdot \Omega^2_X \cap \ZZ^2(\omega)$ and $D(\omega)\subseteq \mathcal{J}\cdot TX$. Then there is an $\eta\in \Omega^1_X$ such that $\theta = \omega\wedge \eta$.
\end{proposition}
\begin{proof}
With the hypotheses we have that, for every $\xi \in D(\omega)$, $i_\xi\theta\in \mathcal{J}^2\cdot \Omega^1_X$.
Then the $f$ such that $i_\xi\theta = f\omega$ must be in $\mathcal{J}$.
In other words $\Phi(\theta)$ is a morphism from $D(\omega)$ to $\mathcal{J}$.
As $Hom(D(\omega),\mathcal{J})= \mathcal{J}\cdot D(\omega)^\vee$ we have that the class of $\Phi(\theta)$ in $\mathcal{E}xt^1_X(N(\omega),\OO_X)$ is $0$, and so is the class of $\theta$ in $H^2(\omega)$.

\end{proof}

\begin{corollary}\label{corDiv}
If $X$ has dimension 2 then for every $\theta\in  \mathcal{J}\cdot \Omega^2_X \cap \ZZ^2(\omega)$ there is an $\eta$ such that $\theta=\omega\wedge \eta$.
\end{corollary}
\begin{proof}
When $X$ is 2-dimensional we have that $D(\omega)$ is generated by a single field $\xi$ and $\sing(\omega) = \sing(\xi)$ as schemes, so $D(\omega)\subseteq \mathcal{J}\cdot TX$.

\end{proof}

\begin{theorem}\label{division}
Let $\omega$ be an integrable 1-form in a smooth variety $X$ and $p\in \sing(\omega)$ be such that $\mathcal{J}(\omega)_p$ is radical and such that $(d\omega)_p\in \mathcal{J}_p\cdot \Omega^2_{X,p}$.
Then there is a formal 1-form $\eta\in\widehat{\Omega^1_{X,p}}$ such that $d\omega=\omega\wedge\eta$.
\end{theorem}
\begin{proof}
The proof proceeds by induction on the dimension of the ambient space.
If $\dim X=2$ the theorem follows by Corollary \ref{corDiv}.
If $\dim X>2$ we consider a \emph{generic} point $\mathfrak{p}\in \sing(\omega)$, then there are two alternatives:

\smallskip
\par If locally around $\mathfrak{p}$ we have $D(\omega)_\mathfrak{p}\subseteq \mathcal{J}_\mathfrak{p}\cdot TX$, then the theorem follows from Proposition \ref{propDiv}.

\smallskip
\par If $D(\omega)\nsubseteq \mathcal{J}_\mathfrak{p}\cdot TX$ then, as $\mathcal{J}_\mathfrak{p}$ is radical, there must be a vector field $\xi\in D(\omega)$ such that $\xi\otimes k(\mathfrak{p})\neq 0$.
If for every such vector field we have $\Phi(d\omega)(\xi)\in \mathcal{J}_\mathfrak{p}$, then we would have $\Phi(d\omega)\in Hom(D(\omega),\mathcal{J}_\mathfrak{p})$ and we would be set.
So we may suppose there is $\xi\in D(\omega)$ such that $\xi\otimes k(\mathfrak{p})\neq 0$ and such that $i_\xi d\omega=f\omega$ with $f\neq 0$ in $k(\mathfrak{p})$.
We can now take a \emph{closed} point $p$ specializing $\mathfrak{p}$ such that $f(p)\neq 0$.
Dividing by $f$ we get a vector field such that
\[
\xi|_p\neq 0,\qquad i_\xi\omega=0,\qquad L_{\xi}\omega=i_\xi d\omega=\omega.
\]
We now take a hyperplane $H$ transversal to $\xi$ at $p$.
And take $\omega_H$ to be the restriction of $\omega$ to $H$.
We can take a formal system of coordinates around $p$, $(x, y_1,\ldots,y_{n-1})$ with $\frac{\partial}{\partial x}=\widehat{\xi}$ and the $y$'s being formal coordinates of $H$ around $p$.
In this coordinate system, as $i_{\frac{\partial}{\partial x}}\omega=0$ we have
\[
\omega = \sum_{i=1}^{n-1} g_i(x,y_1,\ldots,y_{n-1}) dy_i,
\]
as we also have $L_{\frac{\partial}{\partial x}}\omega = \omega$ then the coefficients $g_i$ verify the differential equation $\frac{\partial g_i}{\partial x}=g_i$.
So, for every $i$ there is a series $h_i(y_1,\ldots,y_{n-1})$ such that $g_i(x,y_1,\ldots,y_{n-1})= e^x h_i$.
In other words, for these coordinates we have
\[
\omega=e^x \omega_H.
\]
This implies that if $\mathcal{J}_H$ is ideal of the singular scheme of $\omega_H$ then $\widehat{\mathcal{J}_p}=\mathcal{J}_H[[x]]$.
So $\mathcal{J}_H$ is also radical and moreover $d\omega_H\in \mathcal{J}_H \cdot\widehat{\Omega^2_{S,p}}$.
Then by induction there is a formal $1$-form $\eta'$ such that $d\omega_H=\omega_H\wedge\eta'$.
Now we have, in a formal neighborhood around $p$:
\begin{align*}
d\omega&=d(e^x\omega_H)=e^xdx\wedge \omega_H+e^xd\omega_H=\\
&=e^xdx\wedge \omega_H+e^x\omega_H\wedge\eta'=e^x\omega_H\wedge(-dx+\eta')=\\
&= \omega\wedge\eta,
\end{align*}
where $\eta=-dx+\eta'$.

So far we have proved that, around a generic point $p\in \sing(\omega)$, the class $\overline{d(\omega)}_p$ of $d\omega_p$ in $H^2(\omega)$ is zero.
Then the support of $\overline{d(\omega)}$ must be a proper closed subscheme of the singular locus of $\omega$.
This implies that the support of $H^2(\omega)$, and therefore that of $\mathcal{E}xt^1_X(N(\omega),\OO_X)$, must have as an irreducible component a proper closed set of $\sing(\omega)$.
Again by \cite[Remark 7.7, p.~178]{quallbrunn}, we have the equality $\mathrm{supp}(\mathcal{E}xt^1_X(N(\omega),\OO_X))=\sing(\omega)$, and $\sing(\omega)$ is reduced by hypothesis and therefore have no embedded components.
Hence $\overline{d(\omega)}\in H^2(\omega)$ must be zero, which implies the existence of $\eta$ as in the statement of the theorem.

\end{proof}

\section{Graded projective unfoldings}

Throughout this section we will consider $\omega\in\FF^1(\PP^n,e)$.
First we will define the space of \emph{graded projective unfoldings}, $\UUb(\omega)$, and its related objects which are the ideal of (graded projective) unfoldings $I(\omega)$ and the complex of $S$-modules $R^\bullet(\omega)$.
We will use $I(\omega)$ in Section \ref{teo} to state our main results.

We refer the reader to \cite{moli} for a detailed exposition regarding this subject.

\bigskip

\begin{definition}\label{unf-unf} We define the $S$-module of {\it graded projective unfoldings} of $\omega$ as
\[
\UUb(\omega) = \left\{(h,\eta)\in S\times \Omega^1_{S}: \ L_R(h)\ d\omega = L_R(\omega)\wedge(\eta - dh) \right\}\big/ S.(0,\omega).
\]
For $a\in\NN$, the homogeneous component of degree $a$ can be written as
\[
\UUb(\omega)(a) = \left\{(h,\eta)\in (S\times \Omega^1_{S})(a): \ a\ h\ d\omega = e \ \omega\wedge(\eta - dh) \right\}\big/ S(a-e).(0,\omega).
\]

For $(h,\eta)\in \UUb(\omega)(a)$ and $f\in S(b)$, the graded {\it $S$-module structure} is defined {\it via} the formula
\begin{equation*}
f\cdot (h,\eta) := \left(fh,\ \tfrac{a+b}{a}\  f\eta +\tfrac{1}{a}\ (a\ h\ df- b\ f\ dh)\ \right)\in\UUb(\omega)(a+b).
\end{equation*}
\end{definition}

\medskip

\begin{definition} We define the \emph{isomorphism classes} of graded projective unfoldings, as the quotient $\overline{\UUb}(\omega):=\UUb(\omega)/C_{\UUb}(\omega)$. For $a\in\NN$,  the homogeneous component of degree $a$ of $C_{\UUb}(\omega)$, is defined as
\[
C_{\UUb}(\omega)(a) = \left\{\left(i_X\omega,\frac{1}{e}\left(a\ i_Xd\omega + e\  di_X\omega\right)\right): \ X\in T_S(a-e) \right\}\Big/ S(a-e).(0,\omega).
\]
\end{definition}

\medskip

Let us consider the projection to the first coordinate
\[
\xymatrix@R=0pt{
\UUb(\omega) \ar[r]^{\pi_1} & S.
}
\]

\begin{definition} We define the graded ideals of $S$ associated to $\omega$ as
\begin{align*}
I(\omega) &:= \pi_1(\UUb(\omega)) = \left\{ h\in S:\  h\ d\omega = \omega\wedge\eta\text{ for some } \eta\in\Omega^1_S\right\}\\
J(\omega) &:= \pi_1(C_{\UUb}(\omega)) = \left\{ i_X(\omega)\in S:\ X\in T_S \right\}.
\end{align*}
We will also denote them $I=I(\omega)$ and $J=J(\omega)$ if no confusion arises.
\end{definition}

\begin{proposition}\label{I/J}
The projection $\pi_1:\UUb(\omega)\xrightarrow{} S$ induces the isomorphisms
\[
\overline{\UUb}(\omega)\simeq I(\omega)/J(\omega)\qquad\text{and}\qquad \UUb(\omega)\simeq I(\omega).
 \]
\end{proposition}
\begin{proof}
Let us consider $(h,\eta_1),(h,\eta_2)\in (S\times\Omega^1_S)(a)$ such that
\begin{align*}
a\ hd\omega &= e\ \omega\wedge(\eta_1-dh)\\
a\ hd\omega &= e\ \omega\wedge(\eta_2-dh).
\end{align*}
Then $\omega\wedge (\eta_1-\eta_2) = 0$ and there must exist $f\in S(a-e)$ such that $\eta_1-\eta_2 = f\omega$. This way the classes of $(h,\eta_1)$ and $(h,\eta_2)$ coincide in $\UUb(\omega)$, which shows that $\UUb(\omega)\simeq I(\omega)$. By doing the same for elements of the form $\left(i_X\omega,\frac{a\ i_Xd\omega + e\  di_X\omega}{e}\right)$ we can see the isomorphism $C_\UUb(\omega)\simeq J(\omega)$.

Putting together both arguments we have that $\overline{\UUb}(\omega)\simeq I(\omega)/J(\omega)$.

\end{proof}

\begin{remark}\label{1notinI}
Notice that $1\not\in I$, since the class of $d\omega$ in $H^2(\omega)$ is not zero as we pointed out in Section \ref{foliations}.
\end{remark}
\

Recall from Section \ref{foliations}, that we denote by $\IIc(\eta)$ the ideal of polynomial coefficients of the differential form $\eta$.
\begin{proposition}\label{inc1}
We have the following relations
\[
\IIc(\omega) = J(\omega)\subseteq I(\omega)\ .
\]
\end{proposition}
\begin{proof}
The equality can be easily verified by contracting $\omega$ with the vector fields $\partial/\partial x_i$, $i=0,\ldots,n$.
The inclusion follows from the following fact,
\[
\omega\wedge d\omega=0\Longrightarrow i_X(\omega)d\omega=\omega\wedge (-i_X(d\omega))
\Longrightarrow i_X(\omega)\in I(\omega).
\]

\end{proof}

\

The equivalence between the conditions $\omega\wedge d\omega = 0$ and $d\omega\wedge d\omega = 0$, allows us to define the following complex:

\begin{definition} We define the graded complex  $R^\bullet(\omega)$ of $S$-modules associated to $\omega$, as
\begin{equation*}
\xymatrix@C=30pt{
R^\bullet(\omega): & T_S \ar[r]^-{d\omega\wedge} & \Omega^1_{S} \ar[r]^-{d\omega\wedge} & \Omega^3_{S} \ar[r]^-{d\omega\wedge} & \ldots &
}
\end{equation*}
where $R^s(\omega)=\Omega^{2s-1}_S$ for $s\geq 0$ and the 0-th differential is defined as $d\omega\wedge X:= i_Xd\omega$.
\end{definition}

\

As usual, let us denote by $\ZZ^k(-)$ the cycles of degree $k$ of the given complex. We recall from \cite[Proposition 3.1.5, p.~13 and Theorem 3.2.2, p.~14]{moli}, the following results.

\begin{theorem}
\label{unf-h1} Let $\omega\in\FF^1(\PP^n,e)$, then we have $S$-module isomorphisms
\begin{align*}
\ZZ^1(R^\bullet(\omega))\big/ S.\omega &\ \cong \ \UUb(\omega)\ \cong \ I(\omega)\\
H^1(R^\bullet(\omega)) &\ \cong \ \overline{\UUb}(\omega)\ \cong \ I(\omega)\big/J(\omega).
\end{align*}
\end{theorem}

\medskip

We will use the relation between $\ZZ^1(R^\bullet(\omega))$ and $I$, given by theorem above, to make effective computations of the ideal $I$. Specifically, the unfoldings ideal $I$ can be computed as
\begin{equation}\label{I-comp}
i_R\ZZ^1(R^\bullet(\omega))=I,
\end{equation}
where $i_R$ is the contraction with the radial field. Notice that this is the relation used in the proof of the previous theorem.

\section{The singular set and the unfoldings ideal}\label{teo}

Along this section we will redefine the varieties $\KK_{set}$ and $\LL_{set}$ as projective schemes $\KK$ and $\LL$, respectively. These varieties, together with the unfoldings ideal $I$ defined in the previous section, are our main objects of study. The computations regarding $\KK$ and $\LL$ are going to be done with its ideal of definitions, $K$ and $L$, respectively, regarded as graded ideals over $S$ (see \cite[Chapter II, 5, p.~108]{hart} for the relation between projective schemes and graded modules).

\medskip

In Definition \ref{generic} we state our genericity conditions on a codimension one foliation $\omega$ that we will carry throughout this section.

\medskip

In Theorem \ref{teo1} we prove that the radical of $I$ and the radical of $K$ coincide with mild generic assumptions.

\medskip

In Theorem \ref{teo3} we show that the Kupka scheme $\KK$ equals the Kupka set $\KK_{set}$ and they are non-empty, provided the ideal of $\sing(\omega)$ is radical, $J(\omega)=\sqrt{J(\omega)}$.

\subsection{Definitions}

\begin{definition} For $\omega\in\FF^1(\PP^n,e)$, we define the \emph{Kupka scheme} $\KK(\omega)$ as the scheme theoretic support of $d\omega$ at $\Omega^2_{S}\otimes_S S\big/J(\omega)$. Then, $\KK(\omega)=\proj(S/K(\omega))$ where $K(\omega)$ is the homogeneous ideal defined as
\[
K(\omega)=\ann(\overline{d\omega})+J(\omega)\subseteq S,\quad \overline{d\omega}\in \Omega^2_{S}\otimes_S S\big/J(\omega).
\]
We will denote $\KK=\KK(\omega)$ and $K=K(\omega)$ if no confusion arises.
\end{definition}

\bigskip

For the definition of the scheme $\LL$ we recall the notion of
\emph{ideal quotient} of two $S$-modules $M$ and $N$ as
\[
(N:M) := \left\{a\in S: a.M\subseteq N\right\},
\]
see \cite[Example 1.12, p.~8 and Corollary 3.15, p.~43]{MR0242802} for basic properties. In the case of two ideals $I,J\subseteq S$, we define the \emph{saturation} of $J$ with respect to $I$ as
\[
\left(J:I^\infty\right) := \bigcup_{d\geq 1} \big(J:I^d\big).
\]

Later, we will use the following simple fact.
\begin{lemma}\label{quotient}
Let $J$ be a radical ideal. Then $(J:I)$ is radical and
\[
(J:I) = \left(J:I^\infty\right) = (J:\sqrt{I}).
\]
\end{lemma}

\

One could also define $K(\omega)$ as $K(\omega)=(J\cdot \Omega^2_S: d\omega)$. Then, given that $\Omega^2_S$ is free, we can also write
\begin{equation}\label{Kbis}
K(\omega)=(J:\IIc(d\omega)).
\end{equation}

\

\begin{definition} For $\omega\in\FF^1(\PP^n,e)$, we define the \emph{non-Kupka scheme} $\LL(\omega)$ as the projective scheme $\proj(S/L(\omega))$, where $L(\omega)$ is the homogeneous ideal defined by
\[
L(\omega)=(J(\omega):K(\omega)^{\infty}).
\]
We will write $\LL=\LL(\omega)$ and $L=L(\omega)$ if no confusion arises.
\end{definition}

\bigskip

\begin{remark}\label{reduced}
By Lemma \ref{quotient} and eq. \ref{Kbis} we immediately see that, if $J$ is radical then $K$ and $L$ are radical ideals.
\end{remark}

\bigskip

In the following example we show that the algebraic geometric approach is indeed necessary, since the reduced structure associated to the Kupka scheme $\KK$ differs from the reduced variety associated to $\KK_{set}$. In general, when $J$ is radical, both varieties will coincide, as we will show below.

\begin{example}
Consider the following integrable differential $1$-form $ydx+x^2dy$. Its projectivization in $\PP^2$ is given by
\[
\omega=yz^2dx+x^2zdy-(x^2y+xyz)dz,
\]
and its exterior differential is
\[
d\omega=(2xz-z^2)dx\wedge dy-(2xy+3yz)dx\wedge dz-(2x^2+xz)dy\wedge dz.
\]
In a set-theoretically setting, the singular set of $\omega$ and $d\omega$ are given by
\[
\sing(\omega) = \{(1:0:0),(0:1:0),(0:0:1)\}\qquad \text{and}\qquad \sing(d\omega) = \{(0:1:0)\},
\]
implying that the Kupka set is equal to $\{(1:0:0),(0:0:1)\}$.

The ideal defining $\sing(\omega)$ is
\[
J=(yz^2,x^2z,x^2y+xyz)
\]
giving multiplicities $1$, $4$ and $2$ to the points of $\sing(\omega)$ respectively. The Kupka scheme $\KK$ is defined by the ideal
\[
K=(yz^2,x^2z,2xy-yz).
\]
The support of $K$ is all $\sing(\omega)$ but with multiplicities $1$, $2$ and $2$ respectively.
\qed
\end{example}

\begin{lemma}\label{K=Kset} Let $\omega\in\FF^1(\PP^n,e)$ such that $J=\sqrt{J}$. Then
\[
\KK = \KK_{set}.
\]
\end{lemma}
\begin{proof}
This follows immediately from the equalities
\[
K=(J:\IIc(d\omega))=(J:\IIc(d\omega)^\infty)=\II(\KK_{set}),
\]
where $\II(\KK_{set})$ denotes the (radical) ideal associated to $\KK_{set}$.
\end{proof}

\bigskip

We can now extend the chain of inclusions of Proposition \ref{inc1} by considering the ideal $K$.

\begin{proposition}\label{incl2}
Let $\omega\in\FF^1(\PP^n,e)$. Then, we have the following relations
\[
\IIc(\omega) = J\subseteq I\subseteq K\ .
\]
\end{proposition}
\begin{proof}
We only need to prove the last inclusion.
By definition, given $h\in I$
there exists a differential 1-form $\eta$ such that
\[
h\ d\omega = \omega\wedge \eta.
\]
Then, $h\in (J\cdot \Omega^2_S: d\omega)=K$.

\end{proof}

\subsection{Main results}

Let $\mathfrak{p}$ be a point in $\PP^n$, \emph{i.e.}, a homogeneous prime ideal in $S$ different from the \emph{irrelevant ideal} $(x_0.\ldots,x_n)$, and let $\omega$ be an integrable differential 1-form. We will denote by a subscript $\mathfrak{p}$ the localization at the point $\mathfrak{p}$ and with $\widehat{S}_\mathfrak{p}$ the completion of the local ring $S_\mathfrak{p}$ with respect to the maximal ideal defined by $\mathfrak{p}$.

\bigskip

\begin{definition} We say that $\mathfrak{p}\in\PP^n$ is a \emph{division point of $\omega$} if $1\in I(\omega)_\mathfrak{p}$.
\end{definition}

\bigskip

Recall from Section \ref{foliations} that we refer to the homology of the Koszul complex of $\omega$ as $H^p(\omega)$.

\begin{proposition}\label{division-point}
Let $\mathfrak{p}\in\PP^n$ and let $\omega\in\FF^1(\PP^n,e)$.
The class of $d\omega_\mathfrak{p}$ in $H^2(\omega_\mathfrak{p})$ is zero
if and only if $\mathfrak{p}$ is a division point of $\omega$.

Even more so, assume that formally around $\mathfrak{p}$,
$\widehat{\omega}_\mathfrak{p}\in\Omega^1_S\otimes_S \widehat{S}_\mathfrak{p}$ is equal to $fdg$, where $f,g\in\widehat{S}_\mathfrak{p}$ and $f$ a unit. Then $\mathfrak{p}$ is a division point.
\end{proposition}
\begin{proof}
If $\mathfrak{p}$ is a division point, then $1\in I(\omega)_\mathfrak{p}$, hence $d\omega_\mathfrak{p}=0\in H^2(\omega_\mathfrak{p})$.
Analogously, if $d\omega_\mathfrak{p}=0\in H^2(\omega_\mathfrak{p})$, then $1\in I(\omega)_\mathfrak{p}$.

Assume that $\widehat{\omega}_\mathfrak{p}=fdg$, where $f,g\in\widehat{S}_\mathfrak{p}$ and $f(\mathfrak{p})\neq 0$, then
\[
d\widehat{\omega}_\mathfrak{p}=df\wedge dg=-\frac{1}{f} f dg\wedge df=
\widehat{\omega}_\mathfrak{p}\wedge\big(-\frac{1}{f}df\big)=0\in H^2(\widehat{\omega}_\mathfrak{p}).
\]
Then, $1\in \widehat{I}(\omega)_\mathfrak{p}$, that is, $\widehat{S}_\mathfrak{p}=\widehat{I}(\omega)_\mathfrak{p}$. By Nakayama's Lemma, see \cite[p.~681]{MR1288523}, the inclusion $I(\omega)_\mathfrak{p}\subseteq S_\mathfrak{p}$ is an epimorphism, hence $I(\omega)_\mathfrak{p}=S_\mathfrak{p}$ and $1\in I(\omega)_\mathfrak{p}$.

\end{proof}

\

We now define a subset of the space of foliations on which we are going to state some of our results.

\begin{definition}\label{generic} We define the set $\UU\subseteq\FF^1(\PP^n,e)$ as
\[
\UU = \left\{\omega\in\FF^1(\PP^n,e): \ \forall \mathfrak{p}\not\in\KK(\omega),\,\mathfrak{p}\text{ is a division point of }\omega\right\}.
\]
\end{definition}

\medskip

\begin{remark}\label{rem-u}
A few remarks should be made regarding the set $\UU$:
\begin{enumerate}
\item By 
\cite[Th\'eor\`eme 0.1, p.~163]{m-f1} 
and Proposition \ref{division-point} above, $\UU$ contains the following open subset
\[
\UU' = \left\{\omega\in\FF^1(\PP^n,e):\ \codim(\sing(\omega))\geq 2\text{ and }\codim(\sing(d\omega))\geq 3\right\}.
\]
This open set is one of the usual generic conditions used in the literature.

\item\label{coro-division} If $\sing(\omega)$ is reduced then $\omega\in\UU$.
This follows from Theorem \ref{division} and the simple fact that, in $\PP^n$, the inclusion  $\sing(d\omega)\subseteq\sing(\omega)$ holds. This can be seen by contracting $d\omega$ with the radial field from where we get $i_Rd\omega = e \omega$.

Then, the following set is included in $\UU$,
\[
\UU'':=\{\omega\in\FF^1(\PP^n,e)\ :\ \sing(\omega)\text{ is reduced}\},
\]
A small variation of $\UU''$ is to ask $\sing(\omega)$ to be reduced in the affine cone $\CC^{n+1}$. This is equivalent to asking $J=\sqrt{J}$ and it is slightly stronger, because it removes the irrelevant ideal as an eventual immersed component of $J$. Since our approach is algebraic, we will use this condition as well. We remind that the condition of being reduced is an open condition, see \cite[Th\'eor\`eme (12.2.4), item (v), p.~183]{egaivIII}, then $\UU$ contains the open subset $\UU''$ of $\FF^1(\PP^n,e)$.

\item If $\omega$ admits a global integrating factor $F$ such that
the only components of $\sing(\omega)$ of codimension 2 that intersects $\{F=0\}$ are in $\KK$,
then $\omega\in\UU$. 
Indeed, any codimension 2 component of $\sing(\omega)$ intersects the hypersurface $\{F=0\}$.
Hence, every point $\mathfrak{p}\not\in\KK$ is a division point; see 
\cite[Th\'eor\`eme 0.1, p.~163]{m-f1} and Proposition \ref{division-point}.
This remark is useful for logarithmic foliations.
\end{enumerate}
\end{remark}

The hypothesis $\omega\in \UU$ will be our more general assumption from now on. It is the key to establish relations between the unfolding ideal $I$, the singular ideal $J$ and the Kupka ideal $K$, as the following theorem shows. It gives a global characterization of $\UU$.

\begin{theorem}\label{teo1}
Let $\omega\in\mathcal{U}\subseteq\FF^1(\PP^n,e)$. Then,
\[
\sqrt{I}=\sqrt{K}.
\]
Even more so, if $\sqrt{I}=\sqrt{K}$ then $\omega\in\UU$.
\end{theorem}
\begin{proof}
Take $\omega\in\UU$ and $\mathfrak{p}\notin\KK$. Then $I_\pp = S_\pp$, which is equivalent to $\pp \notin \mathcal{I}$,
where $\mathcal{I} = \proj(S/I)$. This way, we see that $\mathcal{I}_{red}\subseteq \KK$. Reciprocally, if $\mathcal{I}_{red}\subseteq\KK$ then every $\pp\notin\KK$ implies $\pp\notin\mathcal{I}$ and so $I_\pp=S_\pp$, meaning that $\omega\in\UU$. Then have that
\[
\omega\in\UU \iff \mathcal{I}_{red}\subseteq \KK\iff K\subseteq \sqrt{I}\ .
\]
By Proposition \ref{incl2} we already know that $I\subseteq K$. Then, taking radicals in the inclusions $I\subseteq K\subseteq \sqrt{I}$ we have the first implication of the theorem.

\medskip

If now we suppose that $\sqrt{I}=\sqrt{K}$, then $K\subseteq \sqrt{I}$ which is equivalent to $\omega\in\UU$ as we just see.

\end{proof}

\

In Section \ref{app} we will show that in certain components of the space of foliations it can be stated that, generically, $I=K$. The following example shows that this is not always the case; there exists forms
with $\sqrt{I}=\sqrt{K}$, but $I\neq K$.

\begin{example}\label{ex-transverse}
In \cite[5.4, p.~49]{transversely}, the authors find a new irreducible component of $\FF^1(\PP^3,6)$ consisting of foliations with \emph{projective transverse structure}. These foliations can be constructed by considering a differential form $\omega_0$ in $\CC^2$ as
\[
\omega_0 = x_0 dx_1 - x_1 dx_0 + P_2 dx_0 + Q_2 dx_1 + R_2 (x_0 dx_1 - x_1 dx_0)
\]
where $P_2,Q_2$ and $R_2$ are homogeneous polynomials of degree 2. As the article explains, we can consider the homogenization of $\omega_0$, $\Omega_0$, and pullback it by the automorphism of $\CC^3$ given by $\sigma(x_0,x_1,x_2) = (x_0,x_1,x_2+x_0^2)$. This way, we get a new differential 1-form $\omega_1 = \sigma^*(\Omega_0)$ which is not homogeneous. By considering its homogenization again, we finally get an integrable differential form $\omega\in \FF^1(\PP^3,6)$ with projective transverse structure.

Choosing generic polynomials as
\[
P_2 = x_0^2-x_1^2 \qquad Q_2 = x_0^2+x_1^2 \qquad R_2 = x_0^2+x_1^2+x_0 x_1,
\]
and following the process described in \emph{loc. cit.}, we find a generic foliation of such component defined by
{\footnotesize
\begin{align*}
\omega &= \left(-x_0^4 x_1-x_0^4 x_3-2 x_0^3 x_1 x_3+x_0^2 x_1^2 x_3-2 x_0 x_1^3 x_3-2 x_0^2 x_1 x_2 x_3-x_0^2 x_1 x_3^2+x_0 x_1^2 x_3^2+\right.\\
& \left. \hspace{.3cm}-x_1^3 x_ 3^2+ x_0^2 x_2 x_3^2-x_1^2 x_2 x_3^2-x_1 x_2^2 x_3^2\right) dx_0 +  \left(x_0^5+x_0^4 x_3+x_0^2 x_1^2 x_3+2 x_0^3 x_2 x_3+\right.\\
&\hspace{.3cm}\left. +x_0^3 x_3^2-x_0^2 x_1 x _3^2+x_0 x_1^2 x_3^2+x_0^2 x_2 x_3^2+x_1^2 x_2 x_3^2+x_0 x_2^2 x_3^2\right) dx_1+ \\
&\hspace{.3cm}+\left(-x_0^3 x_3^2-x_0^2 x_1 x_3^2+x_0 x_1^2 x_3^2-x_1^3 x_3 ^2\right) dx_2+\left(x_0^5+x_0^4 x_1-x_0^3 x_1^2+x_0^2 x_1^3\right) dx_3.
\end{align*}
}
Making some computations we find that: since $\sqrt{I}=\sqrt{K}$ then $\omega\in\UU$ by Theorem \ref{teo1}\,, $\sing(\omega)$ is not reduced and $I\neq K$.
\qed
\end{example}

In the case of $\PP^2$ we can state the following stronger result.

\begin{lemma}\label{I=KP2} Let $\omega\in\FF^1(\PP^2,e)$, then $I=K$.
\end{lemma}
\begin{proof}
In a similar way to what we did with the definitions of $K$ and $L$ we can characterize the unfoldings ideal of $\omega$ as
\[
I = (\BB^2(\omega):d\omega),
\]
where $\BB^1(\omega)$ are the borders of the differential of the Koszul complex in degree 1. Since $K=(J\cdot\Omega^2_S:d\omega)$, Corollary \ref{corDiv} implies that every $\omega\in\FF^1(\PP^2,e)$ satisfies $I=K$.

\end{proof}

\medskip

We have a similar statement in $\PP^n$ only under certain conditions.

\begin{corollary}\label{technical1}
Let $\omega\in\FF^1(\PP^n,e)$ be such that $J=\sqrt{J}$ and $\KK\cap \LL= \emptyset$. Then
$I = K$.
\end{corollary}
\begin{proof}
By Remark \ref{reduced} we have that $K$ and $L$ are also radical ideals. Let $\pp$ be an \emph{associated prime} of $I$. Using the hypothesis $\KK\cap \LL=\emptyset$ we get that $J_\pp=K_\pp$. Then, by the inclusions $J\subseteq I\subseteq K$ of Property \ref{incl2}, we have that
\[
J_\pp = I_\pp = \mathfrak{q}
\]
where $\mathfrak{q}$ is a $\pp$-primary ideal. Since $J$ is radical, necessarily $\mathfrak{q}=\pp$ and $\pp$ cannot be an embedded prime.
The result now follows from Theorem \ref{teo1}.

\end{proof}

\begin{remark}
In \cite{fmi}, it is shown that the singular locus of generic logarithmic foliations can be decomposed as the disjoint union of the Kupka set and a finite number of isolated points. Even if the authors do not say it, the article also applies to generic rational foliations. In \cite{suwa-multiform} and \cite{suwa-meromorphic}, the unfoldings ideal of generic rational and logarithmic foliations is classified in terms of the functions defining such foliations.
Putting together these works, one can conclude that $I=K$ in these irreducible components.

Also, in Section \ref{app}, we will show that the equality $I=K$ holds generically for pullback and split tangent sheaf foliations. Notice that the assumptions of Corollary \ref{technical1} are verified in all these components we are mentioning.

Following example \ref{ex-transverse} (and many others of the same type that we were able to compute), we believe that it should not be expected that $I=K$ in the component of foliations with projective transverse structure, therefore in $\FF^1(\PP^n,e)$.
\qed
\end{remark}

\medskip

Despite the previous result, the hypotheses $J$ radical and $\KK\cap\LL=\emptyset$ are not necessary to imply $I=K$, as the next two examples shows. For the computations of $I$, see eq. \ref{I-comp}.
\begin{example}
Consider the differential $1$-form in $\PP^2$,
\[
\omega = x_0^2x_2dx_0+x_1^2x_2dx_1+(-x_0^3-x_1^3)dx_2.
\]
The scheme $\sing(\omega)$ consists of three points with multiplicities $1,2$ and $4$.
Also, $\KK$ is the union of the points with multiplicities $1$ and $2$
and $\LL$ is the other point. In this case we have $K\cap L=J$ and $I=K$.
\qed
\end{example}

\

\begin{example}
The family of \textit{Dulac} foliations in $\PP^3$ of type $(p,q)\in\NN^2$, $\mathcal{D}(p,q)$, see \cite[Cap. 1, p.~48]{omegar-libro}, is defined by differential 1-forms as
\[
\omega_{(p,q)} = i_Ri_Yi_X(dx_0\wedge\ldots\wedge dx_3)
\]
where $X$ and $Y$ are vector fields defined as
{\footnotesize
\[
\begin{aligned}
X &= -(q+1)x_0^{p+q-1}x_1\frac{\partial}{\partial x_1} + (p+1)x_0^{p+q-1}x_2 \frac{\partial}{\partial x_2} + \\
& \hspace{.3cm}  + \left[(p-q)x_0^{p+q-1}x_3+ \big((q+1)\beta-(p+1)\alpha\big) x_1^px_2^q\right]\frac{\partial}{\partial x_3}\\
Y &= -\beta x_1\frac{\partial}{\partial x_1} + \alpha x_2\frac{\partial}{\partial x_2} - (p\beta - q\alpha)x_3\frac{\partial}{\partial x_3}
\end{aligned}
\]
}
with $\alpha,\beta\in\CC$ and $R$ is the radial vector field. Note that $[X,Y]=0$.

\

Taking $\alpha=1$ and $\beta=2$ we define the following Dulac foliation of type $(p,q)=(1,1)$ as
\[
\begin{aligned}
\omega_{(1,1)} &= (6x_1^2x_2^2 + 2 x_0x_1x_2x_3) \ dx_0 + (-2 x_0x_1x_2^2-2 x_0^2x_2x_3) \ dx_1 + \\
 & \hspace{.3cm}+ (-4x_0x_1^2x_2-2 x_0^2x_1x_3) \ dx_2 + 2 x_0^2x_1x_2 \ dx_3
\end{aligned}
\]

The scheme $\LL$ is the reduced line $\{x_1=x_2=0\}$. And $\KK$
has $4$ components; two reduced, given by $\{x_1=x_3=0\} \cup \{x_2=x_3=0\}$, and two of multiplicity $2$, given by $\{x_0=x_1=0\}\cup\{x_0=x_2=0\}$.
Despite this pathological situation, we still have  $I=K$.

As one can see in this example the decomposition $\KK\cup\LL$ fails to be $\sing(\omega)$ at a schematic level, since it do not cover all the multiplicities of $\sing(\omega)$. The primary decomposition of $\sing(\omega)$ it is given by 5 components 3 reduced and 2 of multiplicity 4. The two missing components of multiplicity 2, that are the same that $\KK$ has, can be found in the quotient ideal given by $(J(\omega):K(\omega))$.
\qed
\end{example}

\bigskip

From Theorem \ref{teo1} we can draw several results relating, unfoldings and the classes of isomorphism of $\omega$ and with its singular locus and its decomposition in the ideals $K$ and $L$.

\begin{corollary}\label{minimal}
If $\sing(\omega)$ is reduced then the minimal components of $I/J$ and $S/L$ coincide.
\end{corollary}
\begin{proof}
Assume $\sqrt{J}=J$. Then,
\begin{enumerate}
\item $L=(J:K)$ since $K$ is also radical by Remark \ref{reduced}.
\item $(J:I) = (J:\sqrt{I})$ since $J$ is radical by Lemma \ref{quotient}.
\end{enumerate}
By Theorem \ref{teo1}, we know that $\sqrt{I}=\sqrt{K}$ from which we have the following chain of equalities
\[
\ann(I/J) = (J:I) = (J:K) = L=\ann(S/L).
\]
The result follows from \cite[\S 6, Theorem 6.5 (iii), p.~39]{MR879273}.

If $\sing(\omega)$ is reduced, the irrelevant ideal may be an associated prime of $J$.
Then, the minimal associated primes of $S/L$ and of $I/J$ may differ only
by the irrelevant ideal. In any case, the result follows.

\end{proof}

\

\begin{corollary}\label{prop1}
Let $\omega\in\UU$. If $K$ and $\IIc(d\omega)$ are coprime (comaximal), then
\[
I\big/J\cong S\big/ L.
\]
Also, if $\KK\cap\sing(d\omega)=\emptyset$, then the Hilbert polynomial
of $I/J$ and $S/L$ coincide.
\end{corollary}
\begin{proof}
First note that $\IIc(d\omega)\subseteq L$,
\[
 L=(J:K^{\infty})\supseteq(J:K)=(J:(J:\IIc(d\omega))\supseteq \IIc(d\omega).
\]
Then,
\[
 S=K+\IIc(d\omega)\subseteq K+L\subseteq S\Longrightarrow K+L=S.
\]
From \cite[Proposition 1.16, p.~9]{MR0242802} and given that $\sqrt{K}=\sqrt{I}$ we obtain, $I+L=S$.

Second, let us prove that $I\cap L=J$,
\[
J\subseteq I\cap L\subseteq K\cap L=
\bigcup_{n>0}(J:\IIc(d\omega)+K^n)=J.
\]
The equality $K^n+\IIc(d\omega)=S$ for all $n>0$, follows again from \cite[Proposition 1.16, p.~9]{MR0242802}.

Finally,
\[
 I/J=I/I\cap L\cong (I+L)/L=S/L.
\]

The last part follows because the Hilbert polynomial of an ideal
$X$ and $(X:\mathfrak{m}^\infty)$ coincide, where $\mathfrak{m}$ is the irrelevant ideal.
Given that $\KK\cap\sing(d\omega)=\emptyset$, $\KK\cap\LL=\emptyset$ and $\proj(S/I)\cap\LL=\emptyset$.
Also, $K^n+\IIc(d\omega)$ is $\mathfrak{m}$-primary
or equal to $S$.
In any case, $(I\cap L:\mathfrak{m}^\infty)=(J:\mathfrak{m}^\infty)$.
Then,
\[
P_{I/J}=P_{I}-P_{\widehat{J}}=
P_{I}-P_{\widehat{I\cap L}}=
P_{I+L}-P_{L}=
P_{S/L},
\]
where $P_{X}$ is the Hilbert polynomial of the ideal $X$ and $\widehat{X}$ is the saturation of
the ideal $X$.

\end{proof}

\bigskip

In the following example we show that the hypothesis of the previous corollary are necessary.

\begin{example}
Let us consider 3 vector fields with linear coefficients in $\PP^4$, $S,X,Y$, such that
\[
[S,X]=-X\qquad [S,Y]=Y\qquad [X,Y]=2S.
\]
Then $\{S,X,Y\}$ define a Lie algebra isomorphic to $\mathfrak{sl}(2,\CC)$ and the 1-form defined as
\[
\omega = i_Ri_Si_Xi_Y(dx_0\wedge dx_1\wedge dx_2\wedge dx_3)
\]
gives rise to a degree 3 foliation given by the action of $PSL(2,\CC)$ in $\PP^4$, see
\cite[Cap 1, p.~53]{omegar-libro}. Taking
\[
\begin{aligned}
S &= x_0  \frac{\partial}{\partial x_0} - x_1 \frac{\partial}{\partial x_1} + 2  x_2  \frac{\partial}{\partial x_2} - 2 x_3  \frac{\partial}{\partial x_3}\\
X &= x_4\frac{\partial}{\partial x_0}+x_3\frac{\partial}{\partial x_1}+x_0\frac{\partial}{\partial x_2}+x_1\frac{\partial}{\partial x_4}\\
Y &= -4x_2\frac{\partial}{\partial x_0}-6x_4\frac{\partial}{\partial x_1}-4x_1\frac{\partial}{\partial x_3}-6x_0\frac{\partial}{\partial x_4}
\end{aligned}
\]
we get the differential 1-form
{\footnotesize
\[
\begin{aligned}
\omega &= \left(12 x_1^3 x_2-6 x_0^2 x_1 x_3+24 x_0 x_2 x_3^2-4 x_0 x_1^2 x_4-32 x_1 x_2 x_3 x_4+12 x_0 x_3 x_4^2\right) \ dx_0\ + \\
&\hspace{.2cm} + \left(-4 x_0 x_1^2 x_2+18 x_0^3 x_3+16 x_1 x_2^2 x_3-4 x_0^2 x_1 x_4-32 x_0 x_2 x_3 x_4+8 x_1 x_2 x_4^2\right)\  dx_1\ +\\
&\hspace{.2cm} + \left(-8 x_0 x_1^3-4 x_1^2 x_2 x_3-18 x_0^2 x_3^2+28 x_0 x_1 x_3 x_4+8 x_2 x_3^2 x_4+4 x_1^2 x_4^2-12 x_3 x_4^3\right) \ dx_2\ +\\
&\hspace{.2cm} + \left(-12 x_0^3 x_1-12 x_1^2 x_2^2-6 x_0^2 x_2 x_3+28 x_0 x_1 x_2 x_4+8 x_2^2 x_3 x_4+6 x_0^2 x_4^2-12 x_2 x_4^3\right) \ dx_3\ +\\
&\hspace{.2cm} + \left(8 x_0^2 x_1^2+8 x_0 x_1 x_2 x_3-16 x_2^2 x_3^2-12 x_1^2 x_2 x_4-18 x_0^2 x_3 x_4+24 x_2 x_3 x_4^2\right) \ dx_4
\end{aligned}
\]
}

In this situation, $J(\omega)$ is radical. Then,
by Remark \ref{rem-u} \ref{coro-division}, $\omega\in\UU$.
In fact $\sing(\omega)$ has two irreducible components, $\KK$ and $\LL$, both of codimension 2.
Also, $\LL=\sing(d\omega)$, $K=I$ and $\KK\cap \sing(d\omega)\neq\emptyset$,

{\footnotesize
\[
K = I = \left( 2x_0x_1-2x_2x_3-x_4^2,6x_1^2x_2+9x_0^2x_3-18x_2x_3x_4-x_4^3\right),
\]
\[
L= \IIc(d\omega)= \left( x_0x_1x_3+2x_2x_3^2+x_1^2x_4-3x_3x_4^2,2x_1^2x_2-3x_0^2x_3\right.,
\]
\[
2x_0x_1x_2+4x_2^2x_3+3x_0^2x_4-6x_2x_4^2,x_1^3+3x_0x_3^2-3x_1x_3x_4,x_0x_1^2+2x_1x_2x_3-3x_0x_3x_4,
\]
\[
\left.x_0^2x_1+2x_0x_2x_3-2x_1x_2x_4,3x_0^3+4x_1x_2^2-6x_0x_2x_4\right).
\]
}

We can see that $I/J\neq S/L$ by computing the Hilbert polynomials of both graded modules, $P_{I/J}$ and $P_{S/L}$, obtaining
\[
\begin{aligned}
P_{I/J} = 4 P_2 - 11 P_1 + 10 P_0\qquad P_{S/L} = 4 P_2 - 3 P_1
\end{aligned}
\]
where $P_r$ is the Hilbert polynomial of $\PP^r$. Notice that the degree of $P_{S/L}$ shows that $\codim(L) = 2$.
Finally, Corollary \ref{minimal} explains why the highest degree term of both polynomials coincide.
\qed
\end{example}

\bigskip

In this example we show a form $\omega\in\UU$ in $\PP^2$ such that
$K$ and $\IIc(d\omega)$ are not comaximal, but
$\KK\cap\sing(d\omega)=\emptyset$. Generic logarithmic foliations present the same behavior as well, see \cite{fmi}.
\begin{example}
Let us consider a 1-form $\omega$ in $\PP^2$ as
\[
\omega=x_2x_1dx_0+x_2x_0dx_1-(x_0x_1+x_0x_1)dx_2.
\]
The singular ideal is equal to $J=\left( x_1x_2,x_0x_2,x_0x_1\right)$ and it is radical.
Also, $K=\left( x_2,x_0x_1\right)$ and $L=\IIc(d\omega)=\left( x_0,x_1\right)$.
In this case $S/L$ is different from $I/J$, but the Hilbert polynomials coincide.
Note that $\KK$ is equal to the union of two lines, $\{x_0=x_2=0,x_1=x_2=0\}$,
and $\LL$ is equal to another line $\{x_0=x_1=0\}$.
Then, $\KK\cap  \sing(d\omega)=\emptyset$,
but the radical of $K+\IIc(d\omega)$ is the irrelevant ideal $\left( x_0,x_1,x_2\right)$.
From the previous corollary the Hilbert polynomials coincide, specifically, $P_{I/J}=P_{S/L}=1$.
\qed
\end{example}

\bigskip

\begin{remark} Let $\omega\in\mathcal{U}\subseteq\FF^1(\PP^n,e)$.
By Theorem \ref{teo1} we know that $\sqrt{I}=\sqrt{K}$, then there exists a natural number $n$ such that $K^n\subseteq I$. Then
for every $f\in K$, there exists $\eta\in \Omega^1_S$
such that  $f^n d\omega =\omega\wedge\eta$.
Equivalently, $f^n \overline{d\omega}=0$ in $H^2(\omega)$.
\end{remark}

\bigskip

From Theorem \ref{teo1} we can conclude the existence of Kupka points under very general conditions. It is worth mentioning that such result is the first general result on the existence of Kupka points for foliations on $\PP^n$.

\begin{theorem}\label{teo3}
Let $\omega\in\FF^1(\PP^n,e)$ such that $J=\sqrt{J}$.
Then
\[
\KK=\KK_{set}\neq \emptyset.
\]
\end{theorem}
\begin{proof}
Since $\sing(\omega)\neq\emptyset$, the irrelevant ideal of $S$,
$\mathfrak{m}$, can not be an associated prime of $\sqrt{J}$.

Note that $K$ is proper because $1\not\in I$ and $\sqrt{I}=\sqrt{K}=K$, by Remark \ref{1notinI} and Theorem \ref{teo1}, respectively, since $J$ radical implies $K$ radical and $\omega\in\UU$.

If $\KK=\emptyset$, the irrelevant ideal is an associated prime
of $K$, but let us see that any associated prime of $K$ is
an associated prime of $J$. Consider $(K:x)$ an associated prime of $K$. Then,
\[
(K:x)=((J:\IIc(d\omega)):x)=\bigcap_{y\in \IIc(d\omega)} (J:yx)=(J:y_0x),
\]
for some $y_0\in\IIc(d\omega)$. The last equality follows from \cite[Prop. 1.11(ii), p.~8]{MR0242802}
implying that $(J:y_0x)$ is an associated prime of $J$.
Then, the irrelevant ideal is an associated prime of $J=\sqrt{J}$. A contradiction.
Hence, $\KK\neq \emptyset$.

Since $J$ is radical, By Lemma \ref{K=Kset}, we get $\KK=\KK_{set}$ concluding our result.

\end{proof}

\medskip

\begin{proposition}
Let $\omega\in\FF^1(\PP^n,e)$.
If $\KK_{set}\neq \emptyset$, then the reduced Kupka scheme has pure codimension 2.

In particular, if $J=\sqrt{J}$, $\KK$ has pure codimension 2.
\end{proposition}
\begin{proof}
Consider the following sequence of inclusions,
$$\sqrt{\ann(H^2(\omega))}\subseteq \sqrt{I}\subseteq \sqrt{K}\subseteq\big(\sqrt{J}:\IIc(d\omega)\big).$$
By Theorem \ref{koszul-teo}, the first ideal has pure codimension 2
(it follows by localizing $Kosz^\bullet(\omega)$
at the open subset $\codim(\sing(\omega))\geq 3$).
The last ideal is the ideal of the Kupka set which also has pure
codimension 2.

\end{proof}

\bigskip

In our investigation we have noted certain phenomena while looking for examples justifying the hypotheses of our statements. We would like to share with the reader a question we have not been able to settle.

\begin{question}
We do not know any example of an integrable form $\omega$ not in $\UU$. So the question arises:
is it true that $\FF^1(\PP^n,e)=\overline{\UU}$?
\end{question}

\section{Applications}\label{app}

Along this section we describe the unfolding ideals of pullback and split tangent sheaf foliations
in $\FF^1(\PP^n,e)$, see \cite{pullback} and \cite{fj} respectively.

\subsection{Pullback foliations}

In \cite{pullback} the authors prove the generic stability of pullback foliations. We recall from the introduction that a pullback foliation is given by $F^*\omega$, where
$\omega\in\FF^1(\PP^2,e)$  and $\xymatrix@1{F=(F_0:F_1:F_2):\PP^n\ar@{-->}[r]&\PP^2}$ is a rational map, where
$F_i$ is a homogeneous polynomial of degree $\nu$, $i=0,1,2$.

Generic conditions on $\omega$ means that its singular locus is reduced and given by Kupka singularities only; then, $sing(\omega)$ will consist of $N=(e-2)^2+e-1$ different points. Writing $\omega$ as
\[
\omega = A_0\,dx_0+A_1\,dx_0+A_2\, dx_0,
\]
we immediately get
\[
J(\omega)=K(\omega)=(A_0,A_1,A_2).
\]
Regarding the polynomials $F_i$, it is required that the critical values of $F$ be disjoint from the singularities
of $\omega$, as well as the set of critical points be disjoint from
$\{F_0=F_1=F_2=0\}$.

We will call the pair $(F,\omega)$ a \emph{generic pair}, if it satisfies the generic conditions just mentioned.

\medskip

%
%
%

\begin{theorem} Let $(F,\omega)$ be a generic pair. Following the notation above we have that
\[
I(F^*\omega) = K(F^*\omega) = (A_0(F),\, A_1(F),\, A_2(F)).
\]
\end{theorem}
\begin{proof}

By the genericity conditions and Lemma \ref{I=KP2}, we have $I(\omega)=K(\omega)=(A_0,A_1,A_2)$.

Following \cite[p.~700]{pullback}, the Kupka component of $F^*\omega$
is reduced and it is equal to the inverse image of the Kupka component of $\omega$.
Then,
\[
K(F^*\omega)=K_{set}(F^*\omega) = F^*\left(\II(\KK_{set})\right) = \left(A_0(F),A_1(F),A_2(F)\right),
\]
where the first equality follows from Lemma \ref{K=Kset}.

Now, from the inclusion $I\subseteq K$ of Proposition \ref{incl2}, we just need to show that every $A_i(F)\in I(F^*(\omega))$. Given that $A_i\in I(\omega)$, we have
\[
A_id\omega = \omega \wedge(\eta_i-dA_i).
\]
Then, by the commutativity of the exterior differential and the pullback operation,
\[
F^*\left(A_id\omega\right) = F^*\left(\omega\wedge (\eta_i-dA_i)\right)\iff
A_i(F)dF^*\omega = F^*\omega\wedge(F^*\eta_i-dA_i(F)).
\]
Thus, $A_i(F)\in I(F^*(\omega))$.

\end{proof}

\subsection{Foliations with split tangent sheaf}



As first observed in \cite{fj}, several examples of integrable forms on $\PP^n$ are of \emph{split} type,  \emph{e.g.}: such that there are fields $X_1,\ldots, X_{n-1}$ satisfying
\[
\omega =i_R i_{X_1}\cdots i_{X_{n-1}} dx_0\wedge\ldots\wedge dx_n,
\]
where $R$ is the radial field.

Examples treated in \cite{fj} include:
\begin{enumerate}
\item Linear pullbacks:  the pullback of generic degree $e-2$ foliation of $\PP^2$ under a generic linear projection.

\item Foliations associated to affine Lie algebras: These were first studied in \cite{omegar2}.
They are foliations in $\PP^3$ whose tangent sheaf is generated in an affine open set by two vector fields, $X$ and $Y$ satisfying
\[
X = px\frac{\partial}{\partial x}+qy\frac{\partial}{\partial y}+rz\frac{\partial}{\partial z},\qquad  [X,Y]= \ell Y
\]
for some integers $p,q,r,\ell$ with $\mathrm{gcd}(p,q,r)=1$.

\end{enumerate}

As explained in \cite[Section 9]{quallbrunn}, the singular scheme of such foliations is an equidimensional Cohen-Macaulay scheme of codimension 2.
Moreover, in \cite{fj}, in order to establish when these foliations form irreducible components of the space  $\FF^1(\PP^n)$, they require $\omega$ to be in $\UU'$, \emph{c.f.}  Remark \ref{rem-u}.

\medskip

\begin{proposition}
Let $\omega\in \UU'\subseteq \mathcal{F}^1(\PP^n)$ be a foliation of split type.
Then $J=K$.  In particular, for such foliations we have $I=J$.
\end{proposition}
\begin{proof} As $\omega$ is a foliation of split type its singular scheme is an equidimensional Cohen-Macaulay scheme of codimension 2.
Then $d\omega$ does not vanish along any component of $\sing(\omega)$.
In particular, for any associated prime $\mathfrak{p}\in \mathrm{ass}(S/J)$, we have $d\omega_\mathfrak{p}\notin \mathfrak{p}\cdot \Omega^2_{S_{\mathfrak{p}}}$, so $\ann(\overline{d\omega})=(0)\subseteq \Omega^2_S\otimes S/J$.

As any foliation in $\PP^n$ verifies $\sing (d\omega)\subseteq \sing (\omega)$, we must have $L = S$. By Corollary \ref{prop1} this implies $I=J$.

\end{proof}

Recall from the introduction that we say that two unfoldings $\widetilde{\omega}$ and $\widehat{\omega}$ are \emph{isomorphic} if there is an isomorphism $\phi$ of $\PP^n[\varepsilon]$ such that $\phi$ restricts to the identity in the central fiber and $\phi^*\widetilde{\omega}= \widehat{\omega}$.

\begin{corollary}Let $\omega\in\UU'\subseteq\FF^1(\PP^n,e)$ be a foliation of split type. Then, every unfolding of $\omega$ is isomorphic to the trivial unfolding.
\end{corollary}
\begin{proof}
By Theorem \ref{unf-h1} the isomorphism classes of graded unfoldings are parameterized by the quotient $I/J$, which is trivial by the previous proposition.

\end{proof}

\

In this way we can describe $I(\omega)$ for the above examples:

\begin{enumerate}
\item Linear pullbacks: these were treated with more generality in the previous section. In the case of linear pullbacks, the singular locus consists only of Kupka points, which does not need to happen in the general case.
\item Foliations associated to affine Lie algebras: in \cite{omegar2} is shown how the singular set of these foliations is related with the Lie-Klein curves.
These are rational curves $\Gamma_{p,q,r}$ parametrized by $(t:s)\mapsto (t^p:t^qs^{p-q}:t^rs^{p-r}:s^p)$ for integers $p,q,r$ with $\mathrm{gcd}(p,q,r)=1$.

Specifically, in \cite[p.~999]{omegar2} it is shown that the singular scheme 
of foliations given by vector fields $X$ and $Y$ as above with
\[
(p,q,r,\ell)=(\nu^2+\nu+1,\nu+1,1,-1),\qquad \nu \in\mathbb{Z},
\]
is given by the union of the 
Lie-Klein curve $\Gamma_{p,q,r}$, a line and a plane curve of degree $\nu+1$.

\end{enumerate}

\newcommand{\etalchar}[1]{$^{#1}$}

\

\begin{tabular}{l l}
C\'esar Massri$^*$ \hspace{3cm}\null&\textsf{cmassri@dm.uba.ar}\\
Ariel Molinuevo$^*$  &\textsf{amoli@dm.uba.ar}\\
Federico Quallbrunn$^*$  &\textsf{fquallb@dm.uba.ar}\\
&\\
$^*$\textsc{Departamento de Matem\'atica} & \\
\textsc{Pabell\'on I} &\\
\textsc{Ciudad Universitaria} &\\
\textsc{CP C1428EGA} &\\
\textsc{Buenos Aires} &\\
\textsc{Argentina} &
\end{tabular}

\end{document}